\documentclass[12pt,reqno]{amsart}

\usepackage{color}
\usepackage{amsmath, amsthm, amssymb}
\usepackage{amsfonts}
\usepackage[ansinew]{inputenc}
\usepackage[dvips]{epsfig}
\usepackage{graphicx}
\usepackage[english]{babel}
\usepackage{hyperref}
\theoremstyle{plain}
\newtheorem{thm}{Theorem}[section]
\newtheorem{cor}[thm]{Corollary}
\newtheorem{lem}[thm]{Lemma}
\newtheorem{prop}[thm]{Proposition}

\theoremstyle{definition}

\theoremstyle{remark}
\newtheorem{rem}[thm]{Remark}

\numberwithin{equation}{section}

\newcommand{\average}{{\mathchoice {\kern1ex\vcenter{\hrule height.4pt
width 6pt depth0pt} \kern-9.7pt} {\kern1ex\vcenter{\hrule
height.4pt width 4.3pt depth0pt} \kern-7pt} {} {} }}

\def\R{\mathbb{R}}

\begin{document}

\title[Nonexistence results for nonlocal equations]{Nonexistence results for nonlocal equations with critical and supercritical nonlinearities}

\author{Xavier Ros-Oton}

\address{Universitat Polit\`ecnica de Catalunya, Departament de Matem\`{a}tica  Aplicada I, Diagonal 647, 08028 Barcelona, Spain}
\email{xavier.ros.oton@upc.edu}

\thanks{The authors were supported by grants MINECO MTM2011-27739-C04-01 and GENCAT 2009SGR-345.}

\author{Joaquim Serra}

\address{Universitat Polit\`ecnica de Catalunya, Departament de Matem\`{a}tica  Aplicada I, Diagonal 647, 08028 Barcelona, Spain}

\email{joaquim.serra@upc.edu}


\maketitle

\begin{abstract}
We prove nonexistence of nontrivial bounded solutions to some nonlinear problems involving nonlocal operators of the form
\[Lu(x)=\sum a_{ij}\partial_{ij}u+{\rm PV}\int_{\R^n}(u(x)-u(x+y))K(y)dy.\]
These operators are infinitesimal generators of symmetric L\'evy processes.
Our results apply to even kernels $K$ satisfying that
$K(y)|y|^{n+\sigma}$ is nondecreasing along rays from the origin, for some $\sigma\in(0,2)$ in case $a_{ij}\equiv0$ and for $\sigma=2$ in case that $(a_{ij})$ is a positive definite symmetric matrix.

Our nonexistence results concern Dirichlet problems for $L$ in star-shaped domains with critical and supercritical nonlinearities (where the criticality condition is in relation to $n$ and $\sigma$).

We also establish nonexistence of bounded solutions to semilinear equations involving other nonlocal operators such as the higher order fractional Laplacian $(-\Delta)^s$ (here $s>1$) or the fractional $p$-Laplacian.
All these nonexistence results follow from a general variational inequality in the spirit of a classical identity by Pucci and Serrin.
\end{abstract}

\section{Introduction and results}

The aim of this paper is to prove nonexistence results for the following type of nonlinear problems
\begin{equation}\label{pb}
\left\{ \begin{array}{rcll} L u &=&f(x,u)&\textrm{in }\Omega \\
u&=&0&\textrm{in }\mathbb R^n\backslash \Omega,\end{array}\right.\end{equation}
where $\Omega\subset\R^n$ is a bounded domain, $f$ is a critical or supercritical nonlinearity (as defined later), and $L$ is an integro-differential elliptic operator.
Our main results concern operators of the form
\begin{equation}\label{L_K}
L u(x)={\rm PV}\int_{\R^n}\bigl(u(x)-u(x+y)\bigr)K(y)dy
\end{equation}
and
\begin{equation}\label{levy}
Lu(x)=\sum_{i,j}a_{ij}\partial_{ij}u+{\rm PV}\int_{\R^n}\bigl(u(x)-u(x+y)\bigr)K(y)dy,
\end{equation}
where $(a_{ij})$ is a positive definite matrix (independent of $x\in\Omega$) and $K$ is a nonnegative kernel satisfying
\begin{equation}\label{**}
K(y)=K(-y)\quad \textrm{and}\quad \int_{\R^n}\frac{|y|^2}{1+|y|^2}K(y)dy<\infty.
\end{equation}
These operators are infinitesimal generators of symmetric L\'evy processes.

We will state two different nonexistence results, one corresponding to \eqref{L_K} and the other to \eqref{levy}.

On the one hand, we consider operators \eqref{L_K} that may not have a definite order but only satisfy, for some $\sigma\in(0,2)$,
\begin{equation}\label{condicio}
K(y)|y|^{n+\sigma}\quad\textrm{is nondecreasing along rays from the origin.}
\end{equation}
Heuristically, \eqref{condicio} means that even if the order is not defined, $\sigma$ acts as an upper bound for the order of the operator ---see Section \ref{examples} for some examples.
For these operators we prove, under some additional technical assumptions on the kernel, nonexistence of nontrivial bounded solutions to \eqref{pb} in star-shaped domains for supercritical nonlinearities.
When $f(x,u)=|u|^{q-1}u$, the critical power for this class of operators is $q=\frac{n+\sigma}{n-\sigma}$.

On the other hand, we establish the analogous result for second order integro-differential elliptic operators \eqref{levy} with kernels $K$ satisfying \eqref{condicio} with $\sigma=2$.
In this case, the critical power is $q=\frac{n+2}{n-2}$.

Moreover, we can use the same ideas to prove an abstract variational inequality that applies to more general problems.
For instance, we can obtain nonexistence results for semilinear equations involving the higher order fractional Laplacian $(-\Delta)^s$ (i.e., with $s>1$) or the fractional $p$-Laplacian.

When $L$ is the Laplacian $-\Delta$, the nonexistence of nontrivial solutions to \eqref{pb} for critical and supercritical nonlinearities
in star-shaped domains follows from the celebrated Pohozaev identity \cite{P}.
For positive solutions, this result can also be proved with the moving spheres method \cite{Sweers,RZ}.
For more general elliptic operators (such as the $p$-Laplacian, the bilaplacian $\Delta^2$, or $k$-hessian operators), the nonexistence of regular solutions usually follows from Pohozaev-type or Pucci-Serrin identities \cite{PSerrin}.

When $L$ is the fractional Laplacian $(-\Delta)^s$ with $s\in(0,1)$, which corresponds to $K(y)=c_{n,s}|y|^{-n-2s}$ in \eqref{L_K}, this nonexistence result for problem \eqref{pb}
was first obtained by Fall-Weth for positive solutions \cite{FW} (by using the moving spheres method).
In $C^{1,1}$ domains, the nonexistence of nontrivial solutions (not necessarily positive) can be deduced from the Pohozaev identity for the fractional Laplacian,
recently established by the authors in \cite{RS-Poh,RS-CRAS}.

Both the local operator $-\Delta$ and the nonlocal operator $(-\Delta)^s$ satisfy a property of invariance under scaling.
More precisely, denoting $w_{\lambda}(x)=w(\lambda x)$, these operators satisfy $L w_\lambda(x)=\lambda^\sigma Lw(\lambda x)$, with $\sigma=2$ in case $L=-\Delta$ and $\sigma=2s$ in case $L=(-\Delta)^s$.
These scaling exponents are strongly related to the critical powers $q=\frac{n+2}{n-2}$ and $q=\frac{n+2s}{n-2s}$ obtained for power nonlinearities $f(x,u)=|u|^{q-1}u$ in \eqref{pb}.

Here, we prove a nonexistence result for problem \eqref{pb} with operators $L$ that may not satisfy a scale invariance condition
but satisfy \eqref{condicio} instead.
Our arguments are in the same philosophy as Pucci-Serrin \cite{PSerrin}, where they proved a general variational identity that applies to many second order problems.
Here, we prove a variational inequality that applies to the previous integro-differential problems.

Before stating our results recall that, given $\sigma>0$ and $\Omega\subset\R^n$, the nonlinearity $f\in C^{0,1}_{\rm loc}(\overline \Omega\times \R)$ is said to be \emph{supercritical} if
\begin{equation}\label{supercritical}
\frac{n-\sigma}{2}\,t\,f(x,t)>nF(x,t)+x\cdot F_x(x,t)\quad\textrm{for all}\ \ x\in\Omega\ \ \textrm{and}\ \ t\neq0,
\end{equation}
where $F(x,t)=\int_0^tf(x,\tau)d\tau$.
When $f(x,u)=|u|^{q-1}u$, this corresponds to $q>\frac{n+\sigma}{n-\sigma}$.

As explained later on in this Introduction, by bounded solution of \eqref{pb} we mean a critical point $u\in L^\infty(\Omega)$ of the associated energy functional.

Our first nonexistence result reads as follows.
Note that it applies not only to positive solutions but also to changing-sign ones.

In the first two parts of the theorem, we assume the solution $u$ to be $W^{1,r}$ for some $r>1$.
This is a natural assumption that is satisfied when $L$ is a pure fractional Laplacian and also for those operators $L$ with kernels $K$ satisfying an additional assumption on its ``order'', as stated in part (c).

\begin{thm}\label{th1}
Let $K$ be a nonnegative kernel satisfying \eqref{**}, \eqref{condicio} for some $\sigma\in(0,2)$, and
\begin{equation}\label{grad}
K\ \textrm{is}\ C^1(\R^n\setminus\{0\})\ \textrm{and}\ \,|\nabla K(y)|\leq C\,\frac{K(y)}{|y|}\quad \textrm{for all}\ y\neq0
\end{equation}
for some constant $C$.
Let $L$ be given by \eqref{L_K}.
Let $\Omega\subset\R^n$ be any bounded star-shaped domain, and $f\in C^{0,1}_{\rm loc}(\overline \Omega\times \R)$ be a supercritical nonlinearity, i.e., satisfying \eqref{supercritical}.
Let $u$ be any {bounded} solution of \eqref{pb}.
The following statements hold:
\begin{itemize}
\item[(a)] If $u\in W^{1,r}(\Omega)$ for some $r>1$, then $u\equiv0$.

\item[(b)] Assume that $K(y)|y|^{n+\sigma}$ is not constant along some ray from the origin, and that the nonstrict inequality
\begin{equation}\label{supercriticalorcritical}
\frac{n-\sigma}{2}\,t\,f(x,t)\geq nF(x,t)+x\cdot F_x(x,t)\quad\textrm{for all}\ \ x\in\Omega\ \ \textrm{and}\ \ t\in\R
\end{equation}
holds instead of \eqref{supercritical}.
If $u\in W^{1,r}(\Omega)$ for some $r>1$, then $u\equiv0$.

\item[(c)] Assume that in addition $\Omega$ is convex, that the kernel $K$ satisfies
\begin{equation}\label{condicioepsilon}
K(y)|y|^{n+\epsilon}\quad\textrm{is nonincreasing along rays from the origin}
\end{equation}
for some $\epsilon\in(0,\sigma)$, and that
\begin{equation}\label{infsup}
\max_{\partial B_r}K(y)\leq C\min_{\partial B_r}K(y)\quad \textrm{for all}\ r\in(0,1)
\end{equation}
for some constant $C$.
Then, $u\in W^{1,r}(\Omega)$ for some $r>1$, and therefore statements (a) and (b) hold without the assumption $u\in W^{1,r}(\Omega)$.
\end{itemize}
\end{thm}

Note that in part (c) we have the additional assumption that the domain $\Omega$ is convex.
This is used to prove the $W^{1,r}$ regularity of bounded solutions to \eqref{pb} (and it is not needed for example when the operator is the fractional Laplacian, see Remark \ref{remark}).
Note also that condition \eqref{condicio} means in some sense that $L$ has order at most $\sigma$, while \eqref{condicioepsilon} means that $L$ is at least of order $\epsilon$ for some small $\epsilon>0$.

Some examples to which our result applies are sums of fractional Laplacians of different orders, anisotropic operators (i.e., with nonradial kernels), and also operators whose kernels have a singularity different of a power at the origin.
More examples are given in Section \ref{examples}.

Note that for $f(x,u)=|u|^{q-1}u$, part (a) gives nonexistence for supercritical powers $q>\frac{n+\sigma}{n-\sigma}$, while part (b) establishes nonexistence also for the critical power $q=\frac{n+\sigma}{n-\sigma}$.
The nonexistence of nontrivial solutions for the critical power in case that $K(y)|y|^{n+\sigma}$ is constant along all rays from the origin remains an open problem.
Even for the fractional Laplacian $(-\Delta)^s$, this has been only established for positive solutions,
and it is not known for changing-sign solutions.

The existence of nontrivial solutions in \eqref{pb} for subcritical nonlinearities was obtained by Servadei and Valdinoci \cite{SV} by using the mountain pass theorem.
Their result applies to nonlocal operators of the form \eqref{L_K} with symmetric kernels $K$ satisfying $K(y)\geq \lambda |y|^{-n-\sigma}$.

As stated in Theorem \ref{th1}, the additional hypotheses of part (c) lead to the $W^{1,r}(\Omega)$ regularity of bounded solutions for some $r>1$.
This is a consequence of the following proposition.

\begin{prop}\label{regularity}
Let $\Omega\subset \R^n$ be any bounded and convex domain.
Let $L$ be an operator satisfying the hypotheses of Theorem \ref{th1} (c), i.e., satisfying \eqref{L_K}, \eqref{**}, \eqref{condicio}, \eqref{grad}, \eqref{condicioepsilon}, and \eqref{infsup}.
Let $f\in C^{0,1}_{\rm loc}(\overline \Omega\times \R)$, and let $u$ be any bounded solution of \eqref{pb}.
Then,
\begin{equation}\label{solet}
\|u\|_{C^{\epsilon/2}(\R^n)}\leq C\qquad \textrm{and}\qquad |\nabla u(x)|\leq C\delta(x)^{\frac{\epsilon}{2}-1}\ \ \textrm{in}\ \ \Omega,\end{equation}
where $\delta(x)=\textrm{dist}(x,\partial\Omega)$ and $C$ is a constant that depends only on $\Omega$, $\epsilon$, $\sigma$, $f$, and $\|u\|_{L^{\infty}(\Omega)}$.
\end{prop}

Note that \eqref{solet} and the fact that $\Omega$ is convex imply $u\in W^{1,r}(\Omega)$ for all $1<r<\frac{1}{1-\epsilon/2}$.
In \eqref{solet} the exponents $\epsilon/2$ are optimal, as seen when $L=(-\Delta)^{\epsilon/2}$ (see \cite{RS}).

Our second nonexistence result, stated next, deals with operators of the form \eqref{levy}.
Here, the additional assumptions on $\Omega$ and $K$ leading to the $W^{1,r}$ regularity of solutions are not needed thanks to the presence of the second order constant coefficients regularizing term.

\begin{thm}\label{corlevy}
Let $L$ be an operator of the form \eqref{levy},
where $(a_{ij})$ is a positive definite symmetric matrix and $K$ is a nonnegative kernel satisfying \eqref{**}.
Assume in addition that \eqref{grad} holds, and that
\begin{equation}\label{condicio2levy}
K(y)|y|^{n+2}\ \textrm{is nondecreasing along rays from the origin.}
\end{equation}

Let $\Omega\subset\R^n$ be any bounded star-shaped domain, $f\in C^{0,1}_{\rm loc}(\overline \Omega\times \R)$, and $u$ be any bounded solution of \eqref{pb}.
If \eqref{supercriticalorcritical} holds with $\sigma=2$, then $u\equiv0$.
\end{thm}

Note that for $f(x,u)=|u|^{q-1}u$ we obtain nonexistence for critical and supercritical powers $q\geq \frac{n+2}{n-2}$.

The proofs of Theorems \ref{th1} and \ref{corlevy} follow some ideas introduced in our proof of the Pohozaev identity for the fractional Laplacian \cite{RS-Poh}.
The key ingredient in all these proofs is the scaling properties both of the bilinear form associated to $L$ and of the potential energy associated to $f$.
These two terms appear in the variational formulation of \eqref{pb}, as explained next.

Recall that solutions to problem \eqref{pb}, with $L$ given by \eqref{L_K} or \eqref{levy}, are critical points of the functional
\begin{equation}\label{functional}
\mathcal E(u)=\frac12(u,u)-\int_\Omega F(x,u)
\end{equation}
among all functions $u$ satisfying $u\equiv0$ in $\R^n\setminus\Omega$.
Here, $F(x,u)=\int_0^u f(x,t)dt$, and $(\cdot,\cdot)$ is the bilinear form associated to $L$.
More precisely, in case that $L$ is given by \eqref{L_K}, we have
\begin{equation}\label{prodK}
(u,v)=\int_{\R^n}\int_{\R^n}\bigl(u(x)-u(x+y)\bigr)\bigl(v(x)-v(x+y)\bigr)K(y)dx\,dy,
\end{equation}
while in case that $L$ is given by \eqref{levy}, we have
\begin{equation}\label{prodK2}
(u,v)=\int_{\Omega}A(\nabla u,\nabla v)dx+\int_{\R^n}\int_{\R^n}\bigl(u(x)-u(x+y)\bigr)\bigl(v(x)-v(x+y)\bigr)K(y)dx\,dy,
\end{equation}
where $A(p,q)=p^TAq$ and $A=(a_{ij})$ is the matrix in \eqref{levy}.

Both Theorems \ref{th1} and \ref{corlevy} are particular cases of the more general result that we state next.
This result establishes nonexistence of bounded solutions $u\in W^{1,r}(\Omega)$, $r>1$, to problems of the form \eqref{pb} with variational operators $L$ satisfying a scaling inequality.

\begin{prop} \label{th2}
Let $E$ be a Banach space contained in $L^1_{\rm loc}(\R^n)$, and $\|\cdot\|$ be a seminorm in $E$.
Assume that for some $\alpha>0$ the seminorm $\|\cdot\|$ satisfies
\begin{equation}\label{rescale2}
w_\lambda\in E\quad \textrm{and}\quad \|w_\lambda\|\leq\lambda^{-\alpha}\|w\|\ \, \textrm{for every}\ w\in E\ \textrm{and}\ \lambda>1,
\end{equation}
where $w_\lambda(x)=w(\lambda x)$.

Let $\Omega\subset\R^n$ be any bounded star-shaped domain with respect to the origin, $p>1$, and $f\in C^{0,1}_{\rm loc}(\overline \Omega\times \R)$.
Consider the energy functional
\begin{equation}\label{functional2}
\mathcal E(u)=\frac1p\|u\|^p-\int_\Omega F(x,u),
\end{equation}
where $F(x,u)=\int_0^u f(x,t)dt$, and let $u$ be a critical point of $\mathcal E$
among all functions $u\in E$ satisfying $u\equiv0$ in $\R^n\setminus\Omega$.

Assume that $f$ is supercritical, in the sense that
\begin{equation}\label{supercriticalalpha}
\alpha t\,f(x,t)>nF(x,t)+x\cdot F_x(x,t)\quad\textrm{for all}\ \ x\in\Omega\ \ \textrm{and}\ \ t\neq0.
\end{equation}
If $u\in L^\infty(\Omega)\cap W^{1,r}(\Omega)$ for some $r>1$, then $u\equiv0$.
\end{prop}

Some examples to which this result applies are second order variational operators such as the Laplacian or the $p$-Laplacian, the nonlocal operators in Theorems \ref{th1} or \ref{corlevy}, or the higher order fractional Laplacian $(-\Delta)^s$ (here $s>1$).
See Section~\ref{examples} for more examples.

\begin{rem}\label{W1,r}
Proposition \ref{th2} establishes nonexistence of nontrivial bounded  solutions belonging to $W^{1,r}(\Omega)$, $r>1$.
In general, removing the $W^{1,r}$ assumption may be done in two different situations:

First, it may happen that the space $E_\Omega=\{u\in E\,:\, u\equiv0\ \textrm{in}\ \R^n\setminus\Omega\}$ is embedded in $W^{1,r}(\Omega)$, $r>1$.
This happens for instance when considering the natural functional spaces associated to the Laplacian, the $p$-Laplacian with $p>1$, the higher order fractional Laplacian $(-\Delta)^s$ (with $s\geq1$), and of the nonlocal operators considered in Theorem \ref{corlevy}.

Second, even if the space $E_\Omega$ is not embedded in $W^{1,r}$, it is often the case that by some regularity estimates one can prove that critical points of \eqref{functional2} belong to $W^{1,r}$, $r>1$.
This occurs when the operator if the fractional Laplacian, and also in Theorem \ref{th1} (c), thanks to Proposition \ref{regularity}.
\end{rem}

As said before, for local operators of order 2, the nonexistence of regular solutions usually follows from Pohozaev-type or Pucci-Serrin identities \cite{PSerrin}.
Our proofs are in the spirit of these identities.
However, for nonlocal operators this type of identity is only known for the fractional Laplacian $(-\Delta)^s$ with $s\in(0,1)$ \cite{RS-Poh}, and requires a precise knowledge of the boundary behavior of solutions to \eqref{pb} \cite{RS} (that are not available for most $L$).
To overcome this, instead of proving an identity we prove an inequality which is sufficient to prove nonexistence.
This approach allows us to require much less regularity on the solution $u$ and, thus, to include a wide class of operators in our results.

The paper is organized as follows.
In Section \ref{examples} we give a list of examples of operators to which our results apply.
In Section \ref{sketch} we present the main ideas appearing in the proofs of our results.
In Section \ref{sec2} we prove Proposition \ref{th2}.
In Section \ref{sec1} we prove Theorems \ref{th1} and \ref{corlevy}.
Finally, in Section \ref{sec0} we prove Proposition \ref{regularity}.

\section{Examples}
\label{examples}

In this Section we give a list of examples to which our results apply.

\begin{itemize}
\item[(i)] First, note that if $K_1,...,K_m$ are kernels satisfying the hypotheses of Theorem \ref{th1}, and $a_1,...,a_m$ are nonnegative numbers,
then $K=a_1K_1+\cdots+a_mK_m$ also satisfies the hypotheses.
In particular, our nonexistence result applies to operators of the form
\[L=a_1(-\Delta)^{\alpha_1}+\cdots+a_m(-\Delta)^{\alpha_m},\]
with $a_i\geq0$ and $\alpha_i\in(0,1)$.
The critical exponent is $q=\frac{n+2\max\alpha_i}{n-2\max \alpha_i}$.

\item[(ii)] Theorem \ref{th1} may be applied to anisotropic operators $L$ of the form \eqref{L_K} with nonradial kernels such as
\[K(y)=H(y)^{-n-\sigma},\]
where $H$ is any homogeneous function of degree 1 whose restriction to $S^{n-1}$ is positive and $C^1$.
These operators are infinitesimal generators of $\sigma$-stable symmetric L\'evy processes.
The critical exponent is $q=\frac{n+\sigma}{n-\sigma}$.

\item[(iii)] Theorem \ref{th1} applies also to operators with kernels that do not have a power-like singularity at the origin.
For example, the one given by the kernel
\[K(y)=\frac{c}{|y|^{n+\sigma}\log\left(2+\frac{1}{|y|}\right)},\qquad \sigma\in(0,2),\]
whose singularity at $y=0$ is comparable to $|y|^{-n-\sigma}\bigl|\log |y|\bigr|^{-1}$.
In this example we also have that the critical exponent is $q=\frac{n+\sigma}{n-\sigma}$.

Other examples of operators that may not have a definite order are given by infinite sums of fractional Laplacians, such as $L=\sum_{k\geq1}\frac{1}{k^2}(-\Delta)^{s-\frac1k}$.

\item[(iv)] Theorem \ref{corlevy} applies to operators such as $L=-\Delta+(-\Delta)^s$, with $s\in(0,1)$, and also anisotropic operators whose nonlocal part is given by nonradial kernels, as in example~(ii).
For all these operators, the critical power is $q=\frac{n+2}{n-2}$.

\item[(v)] One may take in \eqref{functional2} the $W^{s,p}(\R^n)$ seminorm
    \[\|u\|^p=\int_{\R^n}\int_{\R^n}\frac{|u(x)-u(y)|^p}{|x-y|^{n+ps}}\,dx\,dy.\]
    This leads to nonexistence results for the $s$-fractional $p$-Laplacian operator, considered for example in \cite{Caff-survey,FP}.
    The critical power for this operator is $q=\frac{n+ps}{n-ps}$.

\item[(vi)]
    Our results can also be used to obtain a generalization of Theorem 8 in \cite{PSerrin}, where Pucci and Serrin proved nonexistence results for the bilaplacian $\Delta^2$ and the polylaplacian $(-\Delta)^K$, with $K$ positive integer.
    More precisely, Proposition \ref{th2} can be applied to the $H^s(\R^n)$ seminorm to obtain nonexistence of bounded solutions $u$ to \eqref{pb} with
    $L=(-\Delta)^s$, $s>1$.
    Note that the hypotheses $u\in W^{1,r}(\Omega)$ is always satisfied, since the fractional Sobolev embeddings yield that any function $u\in H^s(\R^n)$ that vanishes outside $\Omega$ belongs to $W^{1,r}(\Omega)$ for $r=2$ (see Remark \ref{W1,r}).

    As an example, when $n>2s$ and $f(u)=\lambda u+|u|^{q-1}u$, one obtains nonexistence of bounded solutions for $\lambda<0$ and $q\geq \frac{n+2s}{n-2s}$ and also for $\lambda\leq0$ and $q> \frac{n+2s}{n-2s}$, as in \cite{PSerrin}.

\item[(vii)] Proposition \ref{th2} can be applied to the usual $W^{1,p}(\Omega)$ norm to obtain nonexistence of bounded weak solutions to \eqref{pb} with $L=-\Delta_p$, the $p$-Laplacian.
These nonexistence results were obtained by Otani in \cite{O} via a Pohozaev-type inequality.

More generally, we may consider nonlinear anisotropic operators that come from setting
\[\|u\|^p=\int_\Omega H(\nabla u)^p |x|^\gamma dx\]
in \eqref{functional2}, where $H$ is any norm in $\R^n$.
In this case, the critical power is $q=\frac{n+\gamma+p}{n+\gamma-p}$.
For $\gamma=0$, some problems involving this class of operators were studied in \cite{BFK,Wulff1,Wulff5}.
For $\gamma\neq0$, nonexistence results for these type of problems were studied in \cite{AP}.

\item[(viii)] From Proposition \ref{th2} one may obtain also nonexistence results for $k$-Hessian operators $S_k(D^2u)$ with $2k<n$.
Recall that $S_k(D^2u)$ are defined in terms of the elementary symmetric polynomials acting on the
eigenvalues of $D^2u$, and that these are variational operators.
In the two extreme cases $k=1$ and $k=n$, we have $S_1(D^2u)=\Delta u$ and $S_n(D^2u) = \textrm{det}\,D^2u$.

Tso studied this problem in \cite{T}, and obtained nonexistence of solutions $u\in C^4(\Omega)\cap C^1(\overline\Omega)$ in smooth star-shaped domains via a Pohozaev identity.
Our results give only nonexistence for supercritical powers $q>\frac{(n+2)k}{n-2k}$, and not for the critical one.
As a counterpart, we only need to assume the solution $u$ to be $L^\infty(\Omega)\cap W^{1,r}(\Omega)$.

\end{itemize}

\section{Sketch of the proof}
\label{sketch}

In this section we sketch the proof of the nonexistence of critical points to functionals of the form
\begin{equation}\label{functional}
\mathcal E(u)=\frac12(u,u)+\int_\Omega F(u),\end{equation}
where $(\cdot,\cdot)$ is a bilinear form satisfying, for some $\alpha>0$,
\begin{equation}\label{rescale}
u_\lambda\in E\quad\textrm{and}\quad \|u_\lambda\|:=(u_\lambda,u_\lambda)^{1/2}\leq\lambda^{-\alpha}(u,u)^{1/2}\ \ \mbox{for all}\ \ \lambda\geq1,
\end{equation}
where $u_\lambda(x)=u(\lambda x)$.
Of course, this is a particular case of Proposition \ref{th2}, in which $p=2$, $E$~is a Hilbert space, and $f$ does not depend on $x$.
Note that in this case condition \eqref{rescale2} reads as \eqref{rescale}.
In case of Theorems \ref{th1} and \ref{corlevy}, the bilinear form is given by \eqref{prodK} and \eqref{prodK2}, respectively.

The proof goes as follows.
Since $u$ is a critical point of \eqref{functional}, then we have that
\[(u,\varphi)=\int_\Omega f(u)\varphi\,dx\quad \textrm{for all}\ \varphi\in E\ \textrm{satisfying}\ \varphi\equiv0\ \textrm{in}\ \R^n\setminus\Omega.\]
Next we use $\varphi=u_\lambda$, with $\lambda>1$, as a test function.
Note that, by \eqref{rescale}, we have $u_\lambda\in E$, and since $\Omega$ is star-shaped, then $u_\lambda\equiv 0$ in $\R^n\setminus\Omega$.
Hence $u_\lambda$ is indeed an admissible test function.
We obtain
\begin{equation}\label{vaca}
(u,u_\lambda)=\int_\Omega f(u)u_\lambda\, dx\ \ \mbox{for all}\ \ \lambda\geq1.
\end{equation}
Now, we differentiate with respect to $\lambda$ in both sides of \eqref{vaca}.
On the one hand, since $u\in L^\infty(\Omega)\cap W^{1,r}(\Omega)$, one can show ---see Lemma \ref{lemaweak}--- that
\[\left.\frac{d}{d\lambda}\right|_{\lambda=1^+}\int_\Omega f(u)u_\lambda\,dx=\int_\Omega (x\cdot\nabla u)f(u)\,dx=\int_\Omega x\cdot\nabla F(u)dx=-n\int_\Omega F(u)dx.\]
On the other hand,
\[\left.\frac{d}{d\lambda}\right|_{\lambda=1^+}(u,u_\lambda)=\left.\frac{d}{d\lambda}\right|_{\lambda=1^+}\left\{\lambda^{-\alpha}I_\lambda\right\}
=-\alpha(u,u)+\left.\frac{d}{d\lambda}\right|_{\lambda=1^+}I_\lambda,\]
where
\begin{equation}\label{I_lambda0}
I_\lambda=\lambda^\alpha(u,u_{\lambda}).
\end{equation}

We now claim that
\begin{equation}\label{0ineq0}
\left.\frac{d}{d\lambda}\right|_{\lambda=1^+}I_\lambda\leq 0.
\end{equation}
Indeed, using \eqref{rescale} and the Cauchy-Schwarz inequality, we deduce
\[I_\lambda\leq \lambda^\alpha\|u\|\|u_{\lambda}\|\leq \|u\|^2=I_1,\]
and thus \eqref{0ineq0} follows.
Therefore, we find
\[-n\int_\Omega F(u)dx=-\alpha\,(u,u)+\left.\frac{d}{d\lambda}\right|_{\lambda=1^+}I_\lambda\leq -\alpha\,(u,u),\]
and since $(u,u)=\int_\Omega uf(u)dx$,
\[\int_\Omega uf(u)dx\leq \frac{n}{\alpha}\int_\Omega F(u)dx.\]
From this, the nonexistence of nontrivial solutions for supercritical nonlinearities follows immediately.

In case of Theorem \ref{th1} (b) and Theorem \ref{corlevy}, with a little more effort we will be able to prove that \eqref{0ineq0} holds with strict inequality, and this will yield the nonexistence result for critical nonlinearities.

When the previous bilinear form is invariant under scaling, in the sense that \eqref{rescale} holds with an equality instead of an inequality, then one has $I_\lambda=(u_{\sqrt{\lambda}},u_{1/\sqrt{\lambda}})$.
In the case $L=(-\Delta)^s$, it is proven in \cite{RS-Poh} that
\[\left.\frac{d}{d\lambda}\right|_{\lambda=1^+} I_\lambda=\Gamma(1+s)\int_{\partial\Omega}\left(\frac{u}{\delta^s}\right)^2(x\cdot\nu)dS,\]
where $\delta(x)=\textrm{dist}(x,\partial\Omega)$.
This gives the boundary term in the Pohozaev identity for the fractional Laplacian.

\begin{rem}
This method can also be used to prove nonexistence results in star-shaped domains with respect to infinity or in the whole space $\Omega=\R^n$.
However, one need to assume some decay on $u$ and its gradient $\nabla u$, which seems a quite restrictive hypothesis.
More precisely, when $f(u)=|u|^{q-1}u$ and the operator is the fractional Laplacian $(-\Delta)^s$, this proof yields nonexistence of bounded solutions (decaying at infinity) for subcritical nonlinearities $q<\frac{n+2s}{n-2s}$ in star-shaped domains with respect to infinity, and for noncritical nonlinearities $q\neq\frac{n+2s}{n-2s}$ in the whole~$\R^n$.
The classification of entire solutions in $\R^n$ for the critical power $q=\frac{n+2s}{n-2s}$ was obtained in \cite{CLO}.
\end{rem}

\section{Proof of Proposition \ref{th2}}
\label{sec2}

In this section we prove Proposition \ref{th2}.
For it, we will need the following lemma, which can be viewed as a H\"older-type inequality in normed spaces.
For example, for $\|u\|=\left(\int_{\R^n} |u|^p\right)^{1/p}$, we recover the usual H\"older inequality (assuming that the Minkowski inequality holds).

\begin{lem}\label{lem}
Let $E$ be a normed space, and $\|\cdot\|$ a seminorm in $E$.
Let $p>1$, and define $\Phi=\frac1p\|\cdot\|^p$.
Assume that $\Phi$ is Gateaux differentiable at $u\in E$, and let $D\Phi(u)$ be the Gateaux differential of $\Phi$ at $u$.
Then, for all $v$ in $E$,
\[\langle D\Phi(u), v\rangle \leq p\,\Phi(u)^{1/p'}\,\Phi(v)^{1/p},\]
where $\frac1p+\frac{1}{p'}=1$.
Moreover, equality holds whenever $v=u$.
\end{lem}

\begin{proof}
Since $\Phi^{1/p}$ is a seminorm, then by the triangle inequality we find that
\[\Phi(u+\varepsilon v)\leq \left\{\Phi(u)^{1/p}+\varepsilon\,\Phi(v)^{1/p}\right\}^p\]
for all $u$ and $v$ in $E$ and for all $\varepsilon\in\R$.
Hence, since these two quantities coincide for $\varepsilon=0$, we deduce
\[\langle D\Phi(u), v\rangle =\left.\frac{d}{d\varepsilon}\right|_{\varepsilon=0}\Phi(u+\varepsilon v)\leq
\left.\frac{d}{d\varepsilon}\right|_{\varepsilon=0}\left\{\Phi(u)^{1/p}+\varepsilon\, \Phi(v)^{1/p}\right\}^p=p\,\Phi(u)^{1/p'}\Phi(v)^{1/p},
\]
and the lemma follows.
\end{proof}

Before giving the proof of Proposition \ref{th2}, we also need the following lemma.

\begin{lem}\label{lemaweak}
Let $\Omega\subset\R^n$ be any bounded domain, and let $u\in W^{1,r}(\Omega)$, $r>1$.
Then,
\[\frac{u_\lambda-u}{\lambda-1}\rightharpoonup x\cdot\nabla u\quad\textrm{weakly in}\ L^1(\Omega),\]
where $u_\lambda(x)=u(\lambda x)$.
\end{lem}

\begin{proof}
Similarly to \cite[Theorem 5.8.3]{E}, it can be proved that
\[\int_\Omega \left|\frac{u_\lambda-u}{\lambda-1}\right|^rdx\leq C\int_\Omega |\nabla u|^rdx.\]
Thus, since $1<r\leq\infty$, then $L^r\cong (L^{r'})'$ and hence there exists a sequence $\lambda_k\rightarrow1$, and a function $v\in L^r(\Omega)$, such that
\[\frac{u_{\lambda_k}-u}{\lambda_k-1}\rightharpoonup v\quad\textrm{weakly in}\ L^r(\Omega).\]

On the other hand note that, for each $\phi\in C^\infty_c(\Omega)$, we have
\[\int_\Omega u\,(x\cdot \nabla \phi)\,dx=-\int_\Omega (x\cdot \nabla u)\,\phi\,dx-n\int_\Omega u\phi\,dx.\]
Moreover, it is immediate to see that, for $\lambda$ sufficiently close to 1,
\[\int_\Omega u\,\frac{\phi_\lambda-\phi}{\lambda-1}\,dx=-\lambda^{-n-1}\int_\Omega\frac{u_{1/\lambda}-u}{1/\lambda-1}\,\phi\,dx+\frac{\lambda^{-n}-1}{\lambda-1}\int_\Omega u\phi\,dx.\]
Therefore,
\[\begin{split}
\int_\Omega u\,(x\cdot\nabla\phi)\,dx&=\lim_{k\rightarrow\infty}\int_\Omega u\,\frac{\phi_{1/\lambda_k}-\phi}{1/\lambda_k-1}\,dx\\
&=\lim_{k\rightarrow\infty} -\int_\Omega\frac{u_{\lambda_k}-u}{\lambda_k-1}\,\phi\, dx-n\int_\Omega u\phi\, dx\\
&= -\int_\Omega v\phi\, dx-n\int_\Omega u\phi\, dx.
\end{split}\]
Thus, it follows that $v=x\cdot\nabla u$.

Now, note that this argument yields also that for each sequence $\mu_k\rightarrow1$ there exists a subsequence $\lambda_k\rightarrow1$ such that
\[\frac{u_{\lambda_k}-u}{\lambda_k-1}\rightharpoonup x\cdot\nabla u\quad\textrm{weakly in}\ L^r(\Omega).\]
Since this can be done for any sequence $\mu_k$, then this implies that
\[\frac{u_{\lambda}-u}{\lambda-1}\rightharpoonup x\cdot\nabla u\quad\textrm{weakly in}\ L^r(\Omega).\]
Finally, since $L^r(\Omega)\subset L^1(\Omega)$, the lemma follows.
\end{proof}

We can now give the:

\begin{proof}[Proof of Proposition \ref{th2}]
Define $\Phi=\frac1p\|\cdot\|^p$.
Since $u$ is a critical point of \eqref{functional2}, then
\begin{equation}\label{test}
\langle D\Phi(u),\varphi\rangle =\int_\Omega f(x,u)\varphi\,dx
\end{equation}
for all $\varphi\in E$ satisfying $\varphi\equiv0$ in $\R^n\setminus\Omega$.
Since $\Omega$ is star-shaped, we may choose $\varphi=u_\lambda$, with $\lambda\geq1$, as a test function in \eqref{test}.
We find
\begin{equation}\label{10}
\langle D\Phi(u),u_\lambda\rangle =\int_\Omega f(x,u)u_\lambda dx\quad \textrm{for all}\ \lambda\geq1.
\end{equation}

We compute now the derivative with respect to $\lambda$ at $\lambda=1^+$ in both sides of \eqref{10}.
On the one hand, using Lemma \ref{lemaweak} we find that
\begin{equation}\begin{split}\label{11}
\left.\frac{d}{d\lambda}\right|_{\lambda=1^+}\int_\Omega u_\lambda f(x,u)dx&=\int_\Omega (x\cdot\nabla u)f(x,u)\,dx\\
&=\int_\Omega \bigl\{x\cdot\nabla \bigl(F(x,u)\bigr)-x\cdot F_x(x,u)\bigr\}dx\\
&=-\int_\Omega \bigl\{n F(x,u)+x\cdot F_x(x,u)\bigr\}dx.
\end{split}\end{equation}
Note that here we have used also that $F(x,u)\in W^{1,1}(\Omega)$, which follows from $u\in L^\infty(\Omega)$, $(x\cdot\nabla u)f(x,u)\in L^r(\Omega)$, and $x\cdot F_x(x,u)\in L^\infty$.

On the other hand, let
\begin{equation}\label{I}
I_\lambda=\lambda^\alpha \langle D\Phi(u), u_{\lambda}\rangle .
\end{equation}
Then,
\begin{equation}\label{12}\begin{split}
\left.\frac{d}{d\lambda}\right|_{\lambda=1^+}\langle D\Phi(u), u_\lambda\rangle &=-\alpha\, \langle D\Phi(u), u\rangle +\left.\frac{d}{d\lambda}\right|_{\lambda=1^+}I_\lambda\\
&=-\alpha \int_\Omega u f(x,u)dx+\left.\frac{d}{d\lambda}\right|_{\lambda=1^+} I_\lambda,
\end{split}
\end{equation}
where we have used that $\langle D\Phi(u), u\rangle =\int_\Omega u f(x,u)dx$, which follows from \eqref{10}.

Now, using Lemma \ref{lem} and the scaling condition \eqref{rescale2}, we find
\[\begin{split}
I_\lambda&=\lambda^\alpha \langle D\Phi(u), u_{\lambda}\rangle
\leq p\,\lambda^\alpha\Phi(u)^{1/p'}\Phi(u_{\lambda})^{1/p}
= \lambda^{\alpha}\|u\|^{p/p'}\|u_\lambda\|\\
&\leq \|u\|^{p/p'+1} = \|u\|^p
=p\,\Phi(u)= \langle D\Phi(u), u\rangle  = I_1,
\end{split}\]
where $1/p+1/p'=1$.
Therefore,
\[\left.\frac{d}{d\lambda}\right|_{\lambda=1^+}I_\lambda\leq 0.\]
Thus, it follows from \eqref{10}, \eqref{11}, and \eqref{12} that
\[-\int_\Omega \bigl\{n F(x,u)+x\cdot F_x(x,u)\bigr\}dx\leq -\alpha \int_\Omega u f(x,u)dx,\]
which contradicts \eqref{supercriticalalpha} unless $u\equiv0$.
\end{proof}

\section{Proof of Theorems \ref{th1} and \ref{corlevy}}
\label{sec1}

This section is devoted to give the

\begin{proof}[Proof of Theorem \ref{th1}]
Recall that $u$ is a weak solution of \eqref{pb} if and only if
\begin{equation}\label{1}
(u,\varphi)=\int_\Omega f(x,u)\varphi\,dx
\end{equation}
for all $\varphi$ satisfying $(\varphi,\varphi)<\infty$ and $\varphi\equiv0$ in $\R^n\setminus\Omega$, where $(\cdot,\cdot)$ is given by \eqref{prodK}.
Note that \eqref{rescale2} is equivalent to \eqref{condicio}.
Thus, part (a) follows from Proposition \ref{th2}, where $\alpha=\frac{n-\sigma}{2}$.

Moreover, it follows from the proof of Proposition \ref{th2} that
\begin{equation}\label{4}
-\int_\Omega \bigl\{n F(x,u)+x\cdot F_x(x,u)\bigr\}dx=\frac{\sigma-n}{2}\int_\Omega uf(x,u)dx+\left.\frac{d}{d\lambda}\right|_{\lambda=1^+}I_\lambda,
\end{equation}
where
\[I_\lambda=\lambda^{\frac{n-\sigma}{2}}\int_{\R^n}\int_{\R^n}\bigl(u(x)-u(x+y)\bigr)\bigl(u_\lambda(x)-u_\lambda(x+y)\bigr)K(y)dx\,dy.\]

Thus, to prove part (b), it suffices to show that
\begin{equation}\label{5}
\left.\frac{d}{d\lambda}\right|_{\lambda=1^+}I_\lambda<0.
\end{equation}
Following the proof of Proposition \ref{th2}, by the Cauchy-Schwarz inequality we find
\[\begin{split}
I_{{\lambda}}&\leq \lambda^{\frac{n-\sigma}{2}}\|u\|\,\|u_\lambda\|\\
&= \sqrt{I_1}\left(\int_{\R^n}\int_{\R^n}\bigl(u(x)-u(x+z)\bigr)^2 \lambda^{-n-\sigma} K(z/\lambda)dx\,dz\right)^{1/2}\\
&= \frac{I_1}{2}+\frac12\int_{\R^n}\int_{\R^n}\bigl(u(x)-u(x+z)\bigr)^2 \lambda^{-n-\sigma} K(z/\lambda)dx\,dz\\
&\leq I_1.
\end{split}\]

Denote now $K(y)=g(y)/|y|^{n+\sigma}$.
Then,
\[\begin{split}
I_1-I_{{\lambda}}&\geq \frac12\int_{\R^n}\int_{\R^n}\bigl(u(x)-u(x+y)\bigr)^2\left\{K(y)-\lambda^{-n-\sigma}K(y/\lambda)\right\}dx\,dy\\
&=\frac12\int_{\R^n}\int_{\R^n}\frac{\bigl(u(x)-u(x+y)\bigr)^2}{|y|^{n+\sigma}}\bigl\{g(y)- g(y/\lambda)\bigr\}dx\,dy,
\end{split}\]
and therefore, by the Fatou lemma
\[-\left.\frac{d}{d\lambda}\right|_{\lambda=1^+}I_\lambda\geq \frac12\int_{\R^n}\int_{\R^n}\frac{\bigl(u(x)-u(x+y)\bigr)^2}{|y|^{n+\sigma}}\,y\cdot\nabla g(y)dx\,dy.\]
Now, recall that $g\in C^1(\R^n\setminus\{0\})$ is nondecreasing along all rays from the origin and nonconstant along some of them.
Then, we have that $y\cdot\nabla g(y)\geq0$ for all $y$, with strict inequality in a small ball $B$.
This yields that
\[\int_{\R^n}\int_{\R^n}\frac{\bigl(u(x)-u(x+y)\bigr)^2}{|y|^{n+\sigma}}\,y\cdot\nabla g(y)dx\,dy>0\]
unless $u\equiv0$.
Indeed, if $u(x)-u(x+y)=0$ for all $x\in \R^n$ and $y\in B$ then $u$ is constant in a neighborhood of $x$, and thus $u$ is constant in all of $\R^n$.

Therefore, using \eqref{4} we find that if $u$ is a nontrivial bounded solution then
\[\frac{n-\sigma}{2}\int_\Omega uf(x,u)dx<\int_\Omega \bigl\{n F(x,u)+x\cdot F_x(x,u)\bigr\}dx,\]
which is a contradiction with \eqref{supercriticalorcritical}.

Finally, part (c) follows from (a), (b), and Proposition \ref{regularity}.
\end{proof}

To end this section, we give the

\begin{proof}[Proof of Theorem \ref{corlevy}]
As explained in the Introduction, weak solutions to problem \eqref{pb} with $L$ given by \eqref{levy} are critical points to \eqref{functional2} with $p=2$ and with
\[\|u\|^2=\int_{\Omega}A(\nabla u,\nabla u)dx+\int_{\R^n}\int_{\R^n}\bigl(u(x)-u(x+y)\bigr)^2K(y)dxdy,\]
where $A(p,q)=p^TAq$ and $A=(a_{ij})$ is the matrix in \eqref{levy}.
It is immediate to see that this norm satisfies \eqref{rescale2} with $\alpha=\frac{n-2}{2}$ whenever \eqref{condicio2levy} holds.
Moreover, since $A$ is positive definite by assumption, then $\|u\|_{W^{1,2}(\Omega)}\leq c\|u\|^2$, and hence $u\in W^{1,r}(\Omega)$ with $r=2$.

Then, it follows from the proof of Proposition \ref{th2} that
\[\frac{n-2}{2}\int_\Omega u f(x,u)dx = \int_\Omega \left\{n F(x,u)+x\cdot F_x(x,u)\right\}dx+\left.\frac{d}{d\lambda}\right|_{\lambda=1^+}I_\lambda,\]
where
\begin{equation}\label{Icorlevy}
\begin{split}
I_\lambda&=\lambda^{\frac{n-2}{2}} \int_\Omega A(\nabla u,\nabla u_\lambda) dx\,+\\
&\qquad+\lambda^{\frac{n-2}{2}}\int_{\R^n}\int_{\R^n}\bigl(u(x)-u(x+y)\bigr)\bigl(u_\lambda(x)-u_\lambda(x+y)\bigr)K(y)dxdy.
\end{split}\end{equation}

Now, as in the proof of Theorem \ref{th1}, we find
\[I_1-I_\lambda\geq \frac12\int_{\R^n}\int_{\R^n}\frac{\bigl(u(x)-u(x+y)\bigr)^2}{|y|^{n+2}}\,\bigl\{g(y)-g(y/\lambda)\bigr\}dy,\]
where $g(y)=K(y)|y|^{n+2}$.
Thus, differentiating with respect to $\lambda$, we find that
\[\left.\frac{d}{d\lambda}\right|_{\lambda=1^+}I_\lambda\geq \frac12\int_{\R^n}\int_{\R^n}\frac{\bigl(u(x)-u(x+y)\bigr)^2}{|y|^{n+2}}\,y\cdot\nabla g(y)dy.\]
Moreover, since $\int_{\R^n}\frac{|y|^2}{1+|y|^2}K(y)dy<\infty$ and $g$ is radially nondecreasing, then it follows that $\lim_{t\rightarrow0}g(t\tau)=0$ for almost all $\tau\in S^{n-1}$.
Thus, if $K$ is not identically zero then $y\cdot\nabla g(y)$ is positive in a small ball $B$, and hence
\[\left.\frac{d}{d\lambda}\right|_{\lambda=1^+}I_\lambda>0\]
unless $u\equiv0$, which yields the desired result.
\end{proof}

\section{Proof of Proposition \ref{regularity}}
\label{sec0}

In this section we prove Proposition \ref{regularity}.
To prove it, we follow the arguments used in \cite{RS}, where we studied the regularity up to the boundary for the Dirichlet problem for the fractional Laplacian.
The main ingredients in the proof of this result are the interior estimates of Silvestre \cite{S} and the supersolution given by the next lemma.

\begin{lem}\label{supersolution}
Let $L$ be an operator of the form \eqref{L_K}, with $K$ symmetric, positive, and satisfying \eqref{condicioepsilon}.
Let $\psi(x)=(x_n)_+^{\epsilon/2}$.
Then,
\[L\psi\geq0\quad\textrm{in}\quad \R^n_+,\]
where $\R^n_+=\{x_n>0\}$.
\end{lem}

\begin{proof}
Assume first $n=1$.
Let $x\in \R_+$.
Since $K$ is symmetric, we have
\[L\psi(x)=\frac12\int_{-\infty}^{+\infty} \bigl(2\psi(x)-\psi(x+y)-\psi(x-y)\bigr)K(y)dy.\]
Then, it is immediate to see that there exists $\rho>0$ such that
\[2\psi(x)-\psi(x+y)-\psi(x-y)>0\quad \textrm{for}\ |y|<\rho\]
and
\[2\psi(x)-\psi(x+y)-\psi(x-y)<0\quad \textrm{for}\ |y|>\rho.\]
Thus, using that $K(y)|y|^{1+\epsilon}$ is nonincreasing in $(0,+\infty)$, and that $(-\Delta)^{\epsilon/2}\psi=0$ in $\R_+$, we find
\[\begin{split}
L\psi(x)&=\frac12\int_{|y|<\rho}\bigl(2\psi(x)-\psi(x+y)-\psi(x-y)\bigr)K(y)dy\\
&\qquad\qquad\qquad+\frac12\int_{|y|>\rho}\bigl(2\psi(x)-\psi(x+y)-\psi(x-y)\bigr)K(y)dy\\
&\geq \frac12\int_{|y|<\rho}\bigl(2\psi(x)-\psi(x+y)-\psi(x-y)\bigr)\frac{K(\rho)|\rho|^{1+\epsilon}}{|y|^{1+\epsilon}}dy\\
&\qquad\qquad\qquad+
\frac12\int_{|y|>\rho}\bigl(2\psi(x)-\psi(x+y)-\psi(x-y)\bigr)\frac{K(\rho)|\rho|^{1+\epsilon}}{|y|^{1+\epsilon}}dy\\
&=K(\rho)|\rho|^{1+\epsilon}\frac12\int_{-\infty}^{+\infty}\frac{2\psi(x)-\psi(x+y)-\psi(x-y)}{|y|^{1+\epsilon}}dy\\
&=K(\rho)|\rho|^{1+\epsilon}(-\Delta)^{\epsilon/2}\psi(x)=0.
\end{split}\]
Thus, the lemma is proved for $n=1$.

Assume now $n>1$, and let $x\in\R^n_+$.
Then,
\begin{equation}\label{radialcoord}\begin{split}
L\psi(x)&=\frac12\int_{\R^n} \bigl(2\psi(x)-\psi(x+y)-\psi(x-y)\bigr)K(y)dy\\
&=\frac14\int_{S^{n-1}} \left(\int_{-\infty}^{+\infty}\bigl(\psi(x)-\psi(x+t\tau)-\psi(x-t\tau)\bigr)t^{n-1}K(t\tau)dt\right)d\tau.
\end{split}\end{equation}
Now, for each $\tau\in S^{n-1}$, the kernel $K_1(t):=t^{n-1}K(t\tau)$ satisfies $K_1(t)t^{1+\epsilon}$ is
nonincreasing in $(0,+\infty)$, and in addition
\[\psi(x+\tau t)=(x_n+\tau_nt)_+^{\epsilon/2}=\tau_n^{\epsilon/2}(x_n/\tau_n+t)_+^{\epsilon/2}.\]
Thus, by using the result in dimension $n=1$, we find
\begin{equation}\label{radialcoord2}
\int_{-\infty}^{+\infty}\bigl(\psi(x)-\psi(x+t\tau)-\psi(x-t\tau)\bigr)t^{n-1}K(t\tau)dt\geq0.
\end{equation}

Therefore, we deduce from \eqref{radialcoord} and \eqref{radialcoord2} that $L\psi(x)\geq0$ for all $x\in \R^n_+$, and the lemma is proved.
\end{proof}

The following result is the analog of Lemma 2.7 in \cite{RS}.

\begin{lem}\label{lema1}
Under the hypotheses of Proposition \ref{regularity}, it holds
\[|u(x)|\leq C\delta(x)^{\epsilon/2}\quad\textrm{for all}\ \,x\in\Omega,\]
where $C$ is a constant depending only on $\Omega$, $\epsilon$, and $\|u\|_{L^\infty(\Omega)}$.
\end{lem}

\begin{proof}
By Lemma \ref{supersolution}, we have that $\psi(x)=(x_n)_+^{\epsilon/2}$ satisfies $L\psi\geq0$ in $\R^n_+$.
Thus, we can truncate this $1D$ supersolution in order to obtain a strict supersolution $\phi$ satisfying
$\phi\equiv\psi$ in $\{x_n<1\}$, $\phi\equiv1$ in $\{x_n>1\}$, and $L\phi\geq c_0$ in $\{0<x_n<1\}$.

We can now use $C\phi$ as a supersolution at each point of the boundary $\partial\Omega$ to deduce $|u|\leq C\delta^{\epsilon/2}$ in $\Omega$;
see Lemma 2.7 in \cite{RS} for more details.
\end{proof}

We next prove the following result, which is the analog of Proposition 2.3 in \cite{RS}.

\begin{prop}
Under the hypotheses of Proposition \ref{regularity},
assume that $w\in L^\infty(\R^n)$ solves $Lw=g$ in $B_1$, with $g\in L^\infty$.
Then, there exists $\alpha>0$ such that
\begin{equation}\label{desired}
\|w\|_{C^\alpha(B_{1/2})}\leq C\left(\|g\|_{L^\infty(B_1)}+\|w\|_{L^\infty(\R^n)}\right),
\end{equation}
where $C$ depends only on $n$, $\epsilon$, $\sigma$, and the constant in \eqref{infsup}.
\end{prop}

\begin{proof}
With slight modifications, the results in \cite{S} yield the desired result.

Indeed, given $\delta>0$ conditions \eqref{condicio}, \eqref{condicioepsilon}, and \eqref{infsup} yield
\begin{equation}\label{cond}
\kappa Lb(x)+2\int_{\R^n\setminus B_{1/4}}\bigl(|8y|^\eta-1\bigr)K(y)dy < \frac12\inf_{A\subset B_2,\ |A|>\delta}\int_A K(y)dy
\end{equation}
for some $\kappa$ and $\eta$ depending only on $n$, $\epsilon$, $\sigma$, and the constant in \eqref{infsup}.
Moreover, since our hypotheses are invariant under scaling, then \eqref{cond} holds at every scale.
Note that \eqref{cond} is exactly hypothesis (2.1) in \cite{S}.

Then, as mentioned by Silvestre in \cite[Remark 4.3]{S}, Lemma 4.1 in \cite{S} holds also with (4.1) therein replaced by $Lw\leq \nu_0$ in $B_1$, with $\nu_0$ depending on $\kappa$.
Therefore, the H\"older regularity of $w$ with the desired estimate \eqref{desired} follows from \cite[Theorem 5.1]{S}.

Note that it is important to have $\sigma$ strictly less than $2$, since otherwise condition \eqref{cond} does not hold.
\end{proof}

The following is the analog of Proposition 2.2 in \cite{RS}.

\begin{prop}
Under the same hypotheses of Proposition \ref{regularity},
assume that $w\in C^\beta(\R^n)$ solves $Lw=g$ in $B_1$, with $g\in C^\beta$, $\beta\in(0,1)$.
Then, there exists $\alpha>0$ such that
\[\|w\|_{C^{\beta+\alpha}(B_{1/2})}\leq C\left(\|g\|_{C^\beta(B_1)}+\|w\|_{C^\beta(\R^n)}\right)\quad \textrm{if}\quad \beta+\alpha<1,\]
\[\|w\|_{C^{0,1}(B_{1/2})}\leq C\left(\|g\|_{C^\beta(B_1)}+\|w\|_{C^\beta(\R^n)}\right)\quad \textrm{if}\quad \beta+\alpha>1,\]
where $C$ and $\alpha$ depend only on $n$, $\epsilon$, $\sigma$, and the constants in \eqref{infsup} and \eqref{grad}.
\end{prop}

\begin{proof}
It follows from the previous Proposition applied to the incremental quotients $(w(x+h)-w(x))/|h|^\beta$ and from Lemma 5.6 in \cite{CC}.
\end{proof}

As a consequence of the last two propositions, we find the following corollaries.
The first one is the analog of Corollary 2.5 in \cite{RS}.

\begin{cor}
\label{correg1}
Under the same hypotheses of Proposition \ref{regularity},
assume that $w\in L^\infty(\R^n)$ solves $Lw=g$ in $B_1$, with $g\in L^\infty$.
Then, there exists $\alpha>0$ such that
\[\|w\|_{C^\alpha(B_{1/2})}\leq C\left(\|g\|_{L^\infty(B_1)}+\|w\|_{L^\infty(B_2)}+\|(1+|y|)^{-n-\epsilon}w(y)\|_{L^1(\R^n)}\right),\]
where $C$ depends only on $n$, $\epsilon$, $\sigma$, and the constants in \eqref{grad} and \eqref{infsup}.
\end{cor}

\begin{proof}
Using \eqref{grad}, the proof is exactly the same as the one in \cite[Corollary 2.5]{RS}.
\end{proof}

The second one is the analog of Corollary 2.4 in \cite{RS}.

\begin{cor}
\label{correg2}
Under the same hypotheses of Proposition \ref{regularity},
assume that $w\in C^\beta(\R^n)$ solves $Lw=g$ in $B_1$, with $g\in C^\beta$, $\beta\in(0,1)$.
Then, there exists $\alpha>0$ such that
\[\|w\|_{C^{\beta+\alpha}(B_{1/2})}\leq C\left(\|g\|_{C^\beta(B_1)}+\|w\|_{C^\beta(\overline{B_2})}+\|(1+|y|)^{-n-\epsilon}w(y)\|_{L^1(\R^n)}\right)\]
if $\beta+\alpha<1$, while
\[\|w\|_{C^{0,1}(B_{1/2})}\leq C\left(\|g\|_{C^\beta(B_1)}+\|w\|_{C^\beta(\overline{B_2})}+\|(1+|y|)^{-n-\epsilon}w(y)\|_{L^1(\R^n)}\right)\]
if $\beta+\alpha>1$.
The constant $C$ depends only on $n$, $\epsilon$, $\sigma$ and the constants in \eqref{grad} and \eqref{infsup}.
\end{cor}

\begin{proof}
Using \eqref{grad}, the proof is the same as the one in \cite[Corollary 2.4]{RS}.
\end{proof}

We can finally give the

\begin{proof}[Proof of Proposition \ref{regularity}]
Let now $x\in\Omega$, and $2R=\textrm{dist}(x,\partial\Omega)$.
Then, one may rescale problem \eqref{pb}-\eqref{L_K} in $B_R=B_R(x)$, to find that $w(y):=u(x+Ry)$ satisfies
$\|w\|_{L^\infty(B_2)}\leq CR^{\epsilon/2}$, $|w(y)|\leq CR^{\epsilon/2}(1+|y|^{\epsilon/2})$ in $\R^n$, and $\|L_Rw\|_{L^\infty(B_1)}\leq C R^\epsilon$,
where
\[L_Rw(y)=\int_{\R^n}\bigl(w(y)-w(y+z)\bigr)K_R(y)dy\]
and $K_R(y)=K(Ry)R^{n+\epsilon}$.

Moreover, it is immediate to check that \eqref{grad} yields
\[|\nabla K_R(y)|\leq C\frac{K_R(y)}{|y|},\]
with the same constant $C$ for each $R\in(0,1)$.
The other hypotheses of Proposition \eqref{regularity} are clearly satisfied by the kernels $K_R$ for each $R\in(0,1)$.

Hence, one may apply Corollaries \ref{correg1} and \ref{correg2} (repeatedly) to obtain
\[|\nabla w(0)|\leq CR^{\epsilon/2}.\]
From this, we deduce that $|\nabla u(x)|\leq CR^{\frac{\epsilon}{2}-1}$, and since this can be done for any $x\in\Omega$, we find
\[|\nabla u(x)|\leq C\delta(x)^{\frac{\epsilon}{2}-1}\quad \textrm{in}\ \Omega,\]
as desired.
The $C^{\epsilon/2}(\R^n)$ regularity of $u$ follows immediately from this gradient bound.
\end{proof}

\begin{rem}\label{remark}
The convexity of the domain has been only used in the construction of the supersolution.
To establish Proposition \ref{regularity} in general $C^{1,1}$ domains, one only needs to construct a supersolution which is not $1D$ but it is radially symmetric and with support in $\R^n\setminus B_1$, as in \cite[Lemma 2.6]{RS}, where it is done for the fractional Laplacian.
\end{rem}

\section*{Acknowledgements}

The authors thank Xavier Cabr\'e for his guidance and useful discussions on the topic of this paper.

\end{document}